\documentclass[final]{article}

\usepackage{graphicx}
\usepackage{amsmath}
\usepackage{amsfonts}
\usepackage{amssymb}
 \usepackage{amscd}

\newtheorem{theorem}{Theorem}[section]

\newtheorem{corollary}[theorem]{Corollary}

\newtheorem{lemma}[theorem]{Lemma}

\newenvironment{proof}[1][Proof]{\textbf{#1.} }{\ \rule{0.5em}{0.5em}}

\title{Generalization of some results concerning eigenvalues
of a certain class of matrices and some
applications\footnote{This work is supported by the Lebanese
university research grants program}}

\author{ Bassam Mourad \thanks{ corresponding
author.Email:bmourad@ul.edu.lb} \\
{\small Department of Mathematics, Faculty of Science V, Lebanese
University, Nabatieh, Lebanon} }

\begin{document}

\maketitle

\begin{abstract} In this note,  we present a generalization of some results
concerning the spectral properties of a certain class of block
matrices. As applications, we study some of its implications on
nonnegative matrices, doubly stochastic matrices and on graph
theory namely on graph spectra and graph energy.
 \end{abstract}

\paragraph*{keywords.}{\footnotesize  eigenvalues, nonnegative matrices, positive matrices, doubly
stochastic matrices, graph theory, graph spectra, graph energy}

\paragraph*{AMS.}  {\footnotesize
 15A12, 15A18, 15A51, 05C50 }

\section{Introduction}
An $n\times n$ matrix with real entries is said to be
\textit{nonnegative } if all of its entries are nonnegative. A
{\em doubly stochastic} matrix is a nonnegative matrix such that
each row and column sum is equal to 1. The Perron-Frobenius
theorem states that if $A$ is a nonnegative matrix, then it has a
nonnegative eigenvalue $r$ (that is the Perron root) which is
greater than or equal to the modulus of each of the other
eigenvalues, and its corresponding  eigenvector $x$ which is
referred to as the Perron-Frobenius eigenvector of $A$ is also
nonnegative. Furthermore, if $A$ is irreducible then $r$ is
positive and the entries of $x$ are strictly positive. In
particular, it is well-known that if $A$ is an $n\times n$ doubly
stochastic matrix then $r=1$ and the corresponding eigenvector is
the column vector $x=e_{n}=\frac{1}{\sqrt{n}}(1,1,...,1)^T\in
\mathbb{R}^n$ where $\mathbb{R}$ denotes the real line. Throughout
this paper, the identity matrix of order $n$ will be denoted by
$I_n.$

A staggering number of papers concerning eigenvalues of
nonnegative and positive matrices as well as doubly stochastic
matrices have appeared (see~\cite{ba,min,pe,sen}). A frequently used auxiliary result appears
in Fiedler  \cite{fi} where the author proves the following
powerful lemma which has been repeatedly used in many situations
particularly in the study of the
nonnegative inverse eigenvalue problem as well as in graph theory (see~\cite{iv,mar}).
\begin{lemma}(\cite{fi})
Let $A$ be an $m\times m$ symmetric matrix with eigenvalues
$\lambda_1,$ $\lambda_2,$... $\lambda_m,$ and let $u$ be the unit
eigenvector corresponding to $\lambda_1.$ Let $B$ be an $n\times
n$ symmetric matrix with eigenvalues $\mu_1,\mu_2,... ,\mu_n$ and
let $v$ be the unit eigenvector corresponding to $\mu_1.$ Then for
any $\rho $, the matrix $ C=\left(
\begin{array}{cc}
A & \rho u v^T \\
\rho v u^T& B \\
\end{array}
\right)$ has eigenvalues $\lambda_2,$... $\lambda_m,$ $\mu_2,$...,
$\mu_n$ and $\gamma_1,$ $\gamma_2$ where $\gamma_1,$ $\gamma_2$
are the eigenvalues of the matrix $\left(
\begin{array}{cc}
\lambda_1 & \rho  \\
\rho & \mu_1 \\
\end{array}
\right).$
\end{lemma}
In~\cite{iv}, the authors generalize the above lemma as follows.
For $j = 1, 2, . . . ,k,$ let $A_j$ be $n_j\times n_j$ symmetric
matrices, with corresponding eigenpairs $(\alpha_{ij} , u_{ij}),$
$i = 1, . . . , n_j.$ Also, for $p = 1, 2, . . . , k-1,$ let
$\rho_p$ be arbitrary constants. In addition,  define the
following tridiagonal by blocks matrix:
 \begin{eqnarray} C=\left(
\begin{array}{ccccc}
A_1                    &  \rho_1 u_{11}u_{12}^T  &       &  &  \\
\rho_1  u_{12}u_{11}^T &    A_2       & \ddots  &\ &  \\
                       &    \ddots              &\ddots     & \ddots &  \\
                       &                         &  &  A_{k-1} & \rho_{k-1} u_{1 k-1}u_{1k}^T \\
                       &                         &  &  \rho_{k-1} u_{1 k}u_{1 k-1}^T  &  A_k\\

\end{array}
\right),
\end{eqnarray}

 \begin{eqnarray} \mbox{ and the tridiagonal matrix }  \widehat{C}=\left(
\begin{array}{ccccc}
\alpha_{11} &  \rho_1            &        &  & \\
\rho_1      &  \alpha_{12}       & \ddots  &  \\
                       &    \ddots              &\ddots     &  &  \\
                       &                         &  &  \alpha_{1 k-1} & \rho_{k-1}  \\
                       &                         &  &  \rho_{k-1}     &  \alpha_{1 k}\\

\end{array}
\right).
\end{eqnarray}
Then they prove the following lemma and use it for an application
in graph theory (see Theorem 5.7 below).
\begin{lemma}(\cite{iv})
 For $j = 1, 2, . . . ,k,$ let $A_j$ be $n_j\times n_j$
symmetric matrices, with corresponding eigenpairs $(\alpha_{ij} ,
u_{ij}),$ $i = 1, . . . , n_j.$ Also, for $p = 1, 2, . . . , k-1,$
let $\rho_p$ be arbitrary constants. Furthermore, suppose that for
each $j$ the system of eigenvectors $u_{ij} , i = 1, . . . , n_j,$
is orthonormal. Also, for $p = 1, 2, . . . , k-1,$ let $\rho_p$ be
arbitrary constants. Then, for any $\rho_1,\rho_2,...,\rho_k,$ the
matrix C in $(1)$ has eigenvalues
$$\alpha_{21}, \alpha_{21}, . . . , \alpha_{n_11}, \alpha_{22},
\alpha_{32}, . . . , \alpha_{n_22}, .. . , \alpha_{2k},
\alpha_{3k}, . . . , \alpha_{n_kk}, \gamma_1, \gamma_2, . . . ,
\gamma_k$$ where $\gamma_1, \gamma_2, . . . , \gamma_k$ are the
eigenvalues of the matrix $ \widehat{C}$ in $(2).$
\end{lemma}

 It is worth noting here that the proofs of the above two lemmas depend
on the fact that the symmetric matrix $C$  has a
complete set of eigenvectors. This
last fact can be dropped  as we shall prove below after presenting
a generalization of the above results by using simpler techniques.
Therefore, our generalization is valid for all square matrices and
not just the symmetric ones. In addition, this generalization is
 particularly useful in applications since finding the
eigenvalues of certain large order matrices depends on computing
the eigenvalues of particular lower order ones as we shall see
below.

This paper is organized as follows. In section 2, we present a
generalization of the above results. Sections 3 and 4 respectively deal with
some applications of this generalization on nonnegative and doubly
stochastic matrices. The last section is concerned with some
applications in graph theory.

\section{Main observations}

We start this section by presenting some auxiliary results that we
are going to use later. The first one is
presented in  Perfect~\cite{pe} and is due to R. Rado.
\begin{theorem}(\cite{pe})\label{th:1} Let $A$ be any $n\times n$ matrix with
eigenvalues $\lambda_1,...,\lambda_n.$ Let $X_1,X_2,...,X_r$ be
$r$ eigenvectors of $A$ corresponding respectively to the
eigenvalues $\lambda_1,...,\lambda_r$ with $r\leq n$ and let
$X=[X_1|X_2|...|X_r]$ be the $n\times r$ matrix whose columns are
$X_1,X_2,...,X_r.$ Then for any $r\times n$  matrix $C,$ the
matrix $A+XC$ has eigenvalues
$\gamma_1,...,\gamma_r,\lambda_{r+1},...,\lambda_n$ where
$\gamma_1,...,\gamma_r$ are the eigenvalues of the matrix
$\Lambda+CX$ where $\Lambda=diagonal(\lambda_1,...,\lambda_r).$
\end{theorem}
In~\cite{sot} the authors presented the following so-called
symmetric version of the above theorem.
\begin{theorem}(\cite{sot})\label{th:2} Let $A$ be any $n\times n$ symmetric matrix with
eigenvalues $\lambda_1,...,\lambda_n.$ Let $\{X_1,X_2,...,X_r\}$
be an orthonormal set eigenvectors of $A$ corresponding
respectively to the eigenvalues $\lambda_1,...,\lambda_r$ with
$r\leq n$ and let $X=[X_1|X_2|...|X_r]$ be the $n\times r$ matrix
whose columns are $X_1,X_2,...,X_r.$ Then for any $r\times r$
symmetric matrix $Y,$ the symmetric matrix $A+XYX^T$ has
eigenvalues $\gamma_1,...,\gamma_r,\lambda_{r+1},...,\lambda_n$
where $\gamma_1,...,\gamma_r$ are the eigenvalues of the matrix
$\Lambda+Y$ where $\Lambda=diagonal(\lambda_1,...,\lambda_r).$
\end{theorem}
Although the authors presented a detailed proof of this result, we notice
 that the preceding theorem can be considered as a special case of Rado's results. Indeed, for any $r\times r$ matrix $Y$ let
 $C=C_Y=YX^T$ in Theorem \ref{th:1}, then  the matrix $A+XC_Y=A+XYX^T$ has eigenvalues
$\gamma_1,...,\gamma_r,\lambda_{r+1},...,\lambda_n$ where
$\gamma_1,...,\gamma_r$ are the eigenvalues of the matrix
$\Lambda+C_YX=\Lambda+YX^TX=\Lambda+YI_r=\Lambda+Y.$

Next we shall use Rado's result to prove a generalization of the
results in the previous section. In particular, we conclude that
the preceding lemmas are easy consequences of Theorem
\ref{th:1}. In addition, we prove that our generalization is
 valid for all matrices not just for the symmetric ones and,
in particular, it is valid for matrices that do not have complete
sets of eigenvectors i.e. non-diagonalizable matrices. Indeed, for
$j = 1, 2, . . . ,k,$ let $A_j$ be any $n_j\times n_j$ matrix with
corresponding eigenvalues
$\lambda_{1j},\lambda_{2j},...,\lambda_{n_jj}.$  For each $j = 1,
2, . . . ,k,$ let $u_j$ be the eigenvector of $A_j$ corresponding
to the eigenvalue $\lambda_{1j}$ with $ \left\Vert
u_{j}\right\Vert =1 .$  Also, for $p = 1, 2, . . . , k$ and $q =
1,2, . . . , k$ let $\rho_{pq}$ be arbitrary constants. In
addition, define the following two  matrices:
\begin{eqnarray} B=\left(
\begin{array}{ccccc}
\qquad A_1+\rho_{11} u_{1}u_{1}^T  & \qquad  \rho_{12} u_{1}u_{2}^T  & \ldots    &  \rho_{1k} u_{1}u_{k}^T \\
\qquad \rho_{21}  u_{2}u_{1}^T &    A_2 +\rho_{22} u_{2}u_{2}^T     & \ldots  &  \rho_{2k} u_{2}u_{k}^T \\
\rho_{31}  u_{3}u_{1}^T &  \rho_{32} u_{3}u_{2}^T      & \ddots& \vdots \\
          \vdots                & \qquad \ddots     & \ddots& \vdots \\
 \rho_{(k-1)1} u_{k-1}u_{1}^T   & \ldots       &  \ldots  & \rho_{(k-1)k} u_{k-1}u_{k}^T  \\
 \rho_{k1} u_{k}u_{1}^T         & \ldots        &  \rho_{k(k-1)} u_{k}u_{k-1}^T  &  A_k+\rho_{kk} u_{k}u_{k}^T \\
\end{array}
\right)
\end{eqnarray}

\begin{eqnarray} \mbox{ and  }   \widehat{B}= \left(
\begin{array}{ccccc}
\lambda_{11}+\rho_{11}  & \qquad \rho_{12}   & \ldots    &  \rho_{1k} \\
 \rho_{21}             &    \lambda_{12} +\rho_{22} & \ldots  &  \rho_{2k} \\
\rho_{31}             & \rho_{32}                  & \ddots& \vdots \\
          \vdots      & \qquad \ddots     & \ddots& \vdots \\
 \rho_{(k-1)1}    & \ldots       &\qquad  \lambda_{1k-1}+\rho_{k-1k-1}  & \rho_{(k-1)k}   \\
 \rho_{k1}          & \ldots        &  \rho_{k(k-1)}  &  \lambda_{1k}+\rho_{kk} \\
\end{array}
\right)
\end{eqnarray}
 The first result of this paper is the following theorem for
which its corresponding symmetric version generalizes the results
of the previous section.
\begin{theorem}
For $j = 1, 2, . . . ,k,$ let $A_j$ be $n_j\times n_j$  matrices
with corresponding eigenvalues
$\lambda_{1j},\lambda_{2j},...,\lambda_{n_jj}$ counted with their
multiplicities. Suppose that for each $j = 1, 2, . . . ,k,$ the
vector  $u_j$ is the eigenvector of $A_j$ corresponding to the
eigenvalue $\lambda_{1j}$ with $ \left\Vert u_{j}\right\Vert =1 .$
Then, for any $\rho_{pq}$ where $1\leq p,q\leq k,$  the matrix $B$
in $(3)$ has eigenvalues
$$\lambda_{21}, \lambda_{31}, . . . , \lambda_{n_11},
\lambda_{22}, \lambda_{32}, . . . , \lambda_{n_22}, .. . ,
\lambda_{2k}, \lambda_{3k}, . . . , \lambda_{n_kk}, \gamma_1,
\gamma_2, . . . , \gamma_k$$ where $\gamma_1, \gamma_2, . . . ,
\gamma_k$ are the eigenvalues of the matrix $ \widehat{B}$ in
$(4).$
\end{theorem}
\begin{proof} First let $n=n_1+n_2+...+n_k$ and also let $0_{n_j}$ be the zero
 column of order $n_j$ whose all components are zeroes for all $j=1, 2, . . .
,k.$ The key point here is to notice that  the $n\times 1$
column vectors $X_1=(u_1^T, 0_{n_2}^T,..., 0_{n_k}^T)^T,$
$X_2=(0_{n_1}^T,u_2^T,0_{n_3}^T,..., 0_{n_k}^T)^T,
...,X_k=(0_{n_1}^T,...,0_{n_{k-1}}^T, u_k^T)^T$ are $k$
eigenvectors of the $n\times n$ matrix $A_1\oplus
A_2\oplus...\oplus A_k.$ Applying Theorem 2.1 with
$X=[X_1|X_2|...|X_k]$ as the $n\times k$ matrix whose columns are
$X_1,X_2,...,X_k,$ and $A$ as the direct sum
given by  $A=A_1\oplus A_2\oplus...\oplus A_k,$ and
with $C$ as the $k\times n$ matrix whose $k$ rows $r_1,$ ...,
$r_k$ are given by the following vectors
$r_1=(\rho_{11}u_{1}^T,\rho_{12}u_{2}^T,...,\rho_{1k}u_{k}^T),
r_2=(\rho_{21}u_{1}^T,\rho_{22}u_{2}^T,\rho_{23}u_{3}^T,...,\rho_{2k}u_{k}^T),...,$
and  $
r_k=(\rho_{k1}u_{1}^T,...,\rho_{k(k-1)}u_{k-1}^T,\rho_{kk}u_{k}^T).$
 Then in this case, a direct verification shows that the matrix $A+XC$ is
equal to the matrix $B$ given in $(3).$  Moreover it can be easily
verified that the matrix
$diagonal(\lambda_{11},...,\lambda_{1k})+CX$ is equal to the
matrix $\widehat{B}$ given in $(4)$  and the proof is complete.
\qquad
\end{proof}\\
As a consequence, we have the following corollary.
\begin{corollary}
For $j = 1, 2, . . . ,k,$ let $A_j$ be $n_j\times n_j$ symmetric
matrices with corresponding eigenvalues
$\lambda_{1j},\lambda_{2j},...,\lambda_{n_jj}$ counted with their
multiplicities. Suppose that for each $j = 1, 2, . . . ,k,$ $u_j$ is
the eigenvector of $A_j$ corresponding to the eigenvalue
$\lambda_{1j}$ with $ \left\Vert u_{j}\right\Vert =1 .$  Then, for
any $\rho_{pq}$ where $1\leq p,q\leq k$ and by taking
$\rho_{qp}=\rho_{pq},$ then the symmetric matrix $B$ in $(3)$ has
eigenvalues
$$\lambda_{21}, \lambda_{31}, . . . , \lambda_{n_11},
\lambda_{22}, \lambda_{32}, . . . , \lambda_{n_22}, .. . ,
\lambda_{2k}, \lambda_{3k}, . . . , \lambda_{n_kk}, \gamma_1,
\gamma_2, . . . , \gamma_k$$ where $\gamma_1, \gamma_2, . . . ,
\gamma_k$ are the eigenvalues of the matrix $ \widehat{B}$ in
$(4).$
\end{corollary}

Now in order to see how the above corollary generalizes Lemma 1.2,
it suffices to take $\rho_{pq}=0$ for all $|p-q|>1$ and
$\rho_{pp}=0$ for all $1\leq p,q\leq k$  in the matrix $B$ given
by (3) to obtain the matrix $C$ in (1) and then clearly in this case
$\widehat{B}$ becomes $ \widehat{C}.$ Finally, we observe that the
case $k=2$ gives the following simple generalization of Lemma 1.1.
\begin{theorem}
Let $A$ be an $m\times m$ symmetric matrix with eigenvalues
$\lambda_1,$ $\lambda_2,$... $\lambda_m,$ and let $u$ be the unit
eigenvector corresponding to $\lambda_1.$ Let $B$ be an $n\times
n$ symmetric matrix with eigenvalues $\mu_1,\mu_2,... ,\mu_n$ and
let $v$ be the unit eigenvector corresponding to $\mu_1.$ Then for
any $\rho, \rho_{11}, \rho_{22} ,$  $ C=\left(
\begin{array}{cc}
A+ \rho_{11}uu^T& \rho u v^T \\
\rho v u^T& B+\rho_{22}vv^T \\
\end{array}
\right)$ has eigenvalues $\lambda_2,$... $\lambda_m,$ $\mu_2,$...,
$\mu_n$ and $\gamma_1,$ $\gamma_2$ where $\gamma_1,$ $\gamma_2$
are the eigenvalues of $\left(
\begin{array}{cc}
\lambda_1 +\rho_{11}& \rho  \\
\rho & \mu_1+\rho_{22} \\
\end{array}
\right).$
\end{theorem}
This last generalization with the special case
$\rho_{11}=\rho_{22}=0$ appears in ~\cite{soto}.

\section{Applications to nonnegative matrices}
In this section, we present some applications of both Theorem 2.3
and Corollary 2.4 on nonnegative matrices. First we shall
introduce the following notation. Let $N_n$ be the set of all
$n$-tuples $\lambda=(\lambda_1;\lambda_2,...,\lambda_n)$ where
$\lambda_2,...,\lambda_n$ are considered unordered and such that
there exists an $n\times n$ nonnegative matrix with spectrum
$\lambda$ and Perron eigenvalue $\lambda_1.$

The first conclusion we can draw from Theorem 2.3 is the
following.
\begin{theorem}
For $j = 1, 2, . . . ,k,$ let $A_j$ be $n_j\times n_j$ nonnegative
matrices with corresponding eigenvalues
$\lambda_{1j},\lambda_{2j},...,\lambda_{n_jj}$ counted with their
multiplicities. Suppose that for each $j = 1, 2, . . . ,k,$ $u_j$ is
the Perron-Frobenius eigenvector of $A_j$ corresponding to the
eigenvalue $\lambda_{1j}$ with $ \left\Vert u_{j}\right\Vert =1 .$
Then, for any $\rho_{pq}\geq 0$ where $1\leq p,q\leq k,$  the
matrix $B$ in $(3)$ is nonnegative and has eigenvalues
$$\lambda_{21}, \lambda_{31}, . . . , \lambda_{n_11},
\lambda_{22}, \lambda_{32}, . . . , \lambda_{n_22}, .. . ,
\lambda_{2k}, \lambda_{3k}, . . . , \lambda_{n_kk}, \gamma_1,
\gamma_2, . . . , \gamma_k$$ where $\gamma_1, \gamma_2, . . . ,
\gamma_k$ are the eigenvalues of the nonnegative matrix $
\widehat{B}$ in $(4).$
\end{theorem}
\begin{proof}
Since $u_j$ is a nonnegative vector for all $j = 1, 2, . . . ,k$
and $\rho_{pq}\geq 0$ where $1\leq p,q\leq k$ then the matrices
$\rho_{pq}u_{p}u_{q}^T $ are all nonnegative for all $1\leq
p,q\leq k.$ Since all the $A_j$ are nonnegative, the proof is
complete.
\end{proof}\\

 The symmetric version of the preceding theorem gives the
following theorem for which the proof is similar.
\begin{theorem}
For $j = 1, 2, . . . ,k,$ let $A_j$ be $n_j\times n_j$ symmetric
nonnegative matrices with corresponding eigenvalues
$\lambda_{1j},\lambda_{2j},...,\lambda_{n_jj}$ counted with their
multiplicities. Suppose that for each $j = 1, 2, . . . ,k,$ $u_j$ is
the Perron-Frobenius eigenvector of $A_j$ corresponding to the
eigenvalue $\lambda_{1j}$ with $ \left\Vert u_{j}\right\Vert =1 .$
Then, for any $\rho_{pq}\geq 0$ where $1\leq p,q\leq k$ and such
that $\rho_{pq}=\rho_{qp},$ the symmetric matrix $B$ in $(3)$ is
nonnegative and has eigenvalues
$$\lambda_{21}, \lambda_{31}, . . . , \lambda_{n_11},
\lambda_{22}, \lambda_{32}, . . . , \lambda_{n_22}, .. . ,
\lambda_{2k}, \lambda_{3k}, . . . , \lambda_{n_kk}, \gamma_1,
\gamma_2, . . . , \gamma_k$$ where $\gamma_1, \gamma_2, . . . ,
\gamma_k$ are the eigenvalues of the nonnegative matrix
$\widehat{B}$ in $(4).$
\end{theorem}

It is worth mentioning here that if the matrices $A_j$ in the
above two theorems are all positive, then from the
Perron-Frobenius theorem all the eigenvectors
 $u_j$ are positive for all $j = 1, 2, . . . ,k.$  Therefore for all $\rho_{pq}> 0$
 where $1\leq p,q\leq k,$ the matrices $B$ and $ \widehat{B}$ are also positive.
On the other hand, if all the $\rho_{pp}$ are zeroes and all the
matrices $A_j$ are of trace zero then $B$ has zero trace. Thus an obvious application
 for the above two theorems, lies
in the study of the inverse eigenvalue problem for
nonnegative matrices namely in the study of the so-called {\it
extreme nonnegative matrices} (see \cite{la, laf}) where
trace-zero nonnegative matrices are such ones and positive
matrices are not (see also \cite{bo}). More precisely, such extreme points in lower dimensions can be used in the above two theorems to generate extreme points in higher dimension.

 Recall that all the $\rho_{ij}$s in the above two theorems are nonnegative. Therefore  by seeking a convenient choice of the $\rho_{ij}$s
  that makes the matrix $\widehat{B}$ has a special structure such as  diagonal, upper-triangular, or any other form that makes the computations of the eigenvalues of  $\widehat{B}$  possible,  is helpful for obtaining some
   useful consequences of Theorem 3.2. It should be noted that the choice which makes $\widehat{B}$ diagonal, results in well-known conclusions. However,  a choice of the $\rho_{ij}$s that makes the matrix $\widehat{B}$  {\it circulant} for example, gives the following new result.
 \begin{theorem}
For $j = 1, 2, . . . ,k,$ let $(\lambda_{1j}; \lambda_{2j}, . . .
, \lambda_{n_jj})\in N_{n_j}.$ Without loss of
generality, suppose that $\lambda_{11}\geq \underset{j}{\max}
(\lambda_{1j}).$  For any nonnegative reals  $\rho_{11},\rho_{12},...,\rho_{1k},$
 define the polynomial $p(x)=(\lambda_{11}+\rho_{11})+\rho_{12}x+...+\rho_{1k}x^{k-1}.$ Then the $(n_1+...+n_k)$-tuple
$$(\underbrace{p(1);\lambda_{21}, . . . , \lambda_{n_11}},
\underbrace{p(w),\lambda_{22}, . . . ,
\lambda_{n_22}}, .. . , \underbrace{p(w^{k-1}),\lambda_{2k}, . . . , \lambda_{n_kk}})$$ is in $N_{n_1+...+n_k},$ where $w$ is a $k$th primitive root
 of unity i.e. $w^k=1$ and $w^l\neq 1$ for any integer $l$ less than $k.$
\end{theorem}
\begin{proof} For $j = 1, 2, . . . ,k,$ let $A_j$ realizes $(\lambda_{1j}; \lambda_{2j}, . . ., \lambda_{n_jj}).$ Next, we choose the nonnegative reals $\rho_{ij}$ for $i=2,...,k$ and $j=1,...,k$ in terms of $\rho_{11},\rho_{12},...,\rho_{1k}$
so that $\widehat{B}$ is a nonnegative circulant with first row $(\lambda_{11}+\rho_{11},\rho_{12},...,\rho_{1k}).$ Moreover, the eigenvalues of the circulant matrix $\widehat{B}$ are well-known and given by $p(w^l)$ for $l=0,1,...,k-1$ (see~\cite{dav}). To complete the proof, we apply Theorem 3.2 with those matrices $A_1, A_2,...,A_k$ and the chosen $\rho_{ij}$  to obtain the nonnegative matrix $B$ that has the required spectrum.
\end{proof}

\section{Applications to doubly stochastic matrices}
 In this section, we study some of the effects of Theorem 2.3 on the spectral properties
  of doubly stochastic matrices which has been the object of study for a long time
  (see \cite{hw,mo,mou,mour,moura,per,re} and the reference therein).

  First let us recall that in ~\cite{mourad}, we used an
  extension of Fiedler's lemma to prove the following.
\begin{theorem}
Let $T_1$ be an $m\times m$ diagonalizable doubly stochastic
matrix with eigenvalues $1,$ $\lambda_2,$... $\lambda_m,$ and let
$T_2$ be an $n\times n$ diagonalizable doubly stochastic matrix
with eigenvalues $1,$ $\mu_2,$... $\mu_n.$  For any nonnegative reals $\alpha$ and $\rho$  that do not
vanish simultaneously, and for $n\geq m,$   the $(m+n)\times (m+n)$
matrix $D$ given by: $ D=\frac{1}{\alpha+\frac{\rho n
}{\sqrt{mn}}}\left(
\begin{array}{cc}
\alpha T_1 & \rho e_m e_n^T \\
\rho e_n e_m^T& (\alpha+\rho\frac{n-m}{\sqrt{mn}})T_2 \\
\end{array}
\right)$ is doubly stochastic with eigenvalues given by: $1,
\frac{\alpha\sqrt{mn}-\rho m}{\alpha\sqrt{mn}+\rho n},
\frac{\alpha}{\alpha+\frac{\rho n }{\sqrt{mn}}}\lambda_2,..,
\frac{\alpha}{\alpha+\frac{\rho n }{\sqrt{mn}}}\lambda_m,
\frac{\alpha\sqrt{mn}+\rho (n-m)}{\alpha\sqrt{mn}+\rho n}\mu_2,..,
\frac{\alpha\sqrt{mn}+\rho (n-m)}{\alpha\sqrt{mn}+\rho n}\mu_n.$
\end{theorem}

Now in view of Theorem 2.3, the diagonalizability condition in the
previous theorem can be removed and therefore we have the following
more general result.
\begin{theorem}
Let $T_1$ be an $m\times m$ doubly stochastic matrix with
eigenvalues $1,$ $\lambda_2,$... $\lambda_m,$ and let $T_2$ be an
$n\times n$ doubly stochastic matrix with eigenvalues $1,$
$\mu_2,$... $\mu_n.$  Moreover suppose that $n\geq m$ then for any
$\alpha \geq 0$ and for any $\rho \geq 0$ such that $\alpha$ and
$\rho$ do not vanish simultaneously,  the $(m+n)\times (m+n)$
matrix $D$ defined by $ D=\frac{1}{\alpha+\frac{\rho n
}{\sqrt{mn}}}\left(
\begin{array}{cc}
\alpha T_1 & \rho e_m e_n^T \\
\rho e_n e_m^T& (\alpha+\rho\frac{n-m}{\sqrt{mn}})T_2 \\
\end{array}
\right)$ is doubly stochastic with eigenvalues given by: $1,
\frac{\alpha\sqrt{mn}-\rho m}{\alpha\sqrt{mn}+\rho n},
\frac{\alpha}{\alpha+\frac{\rho n }{\sqrt{mn}}}\lambda_2,..,
\frac{\alpha}{\alpha+\frac{\rho n }{\sqrt{mn}}}\lambda_m,
\frac{\alpha\sqrt{mn}+\rho (n-m)}{\alpha\sqrt{mn}+\rho n}\mu_2,..,
\frac{\alpha\sqrt{mn}+\rho (n-m)}{\alpha\sqrt{mn}+\rho n}\mu_n.$
\end{theorem}

\begin{proof} In Theorem 2.3, if we let $k=2,$ $n_1=m$ and  $n_2=n,$ then it suffice to take
 $A_1=\alpha T_1,$ $A_2=(\alpha+\rho\frac{n-m}{\sqrt{mn}})T_2,$  $\rho_{12}=\rho_{21}=\rho$
 and $\rho_{11}=\rho_{22}=0$ in the matrix $B$ of $(3)$ to
 complete the proof.
\end{proof}\\
Clearly in the above theorem if $\rho=0,$ then obviously
$D=T_1\oplus T_2.$ Also if $\rho\neq 0$ and $\alpha \neq 0$ and
$T_1$ and $T_2$ are positive then the matrix $D$ is necessarily
positive. On the other hand, if $A$ and $B$ are of trace-zero or
$\alpha =0$ and $T_2$ is of trace-zero, then $D$ is of trace-zero.
As mentioned earlier for the nonnegative inverse eigenvalue problem and in a similar fashion,
 the preceding two theorems have obvious applications in the study of the inverse eigenvalue
problem for doubly stochastic matrices (see \cite{mourad}).

Another application of Theorem 2.3 with different input gives
another useful result concerning the spectral properties of doubly stochastic matrices.
\begin{theorem}
Let $T_1$ be an $m\times m$ doubly stochastic matrix with
eigenvalues $1,$ $\lambda_2,$... $\lambda_m,$ and let $T_2$ be an
$n\times n$ doubly stochastic matrix with eigenvalues $1,$
$\mu_2,$... $\mu_n$ each counted with their multiplicities.
Without loss of generality, suppose that $n\geq m,$ then for any
nonnegative numbers  $\alpha$ and $\rho,$ the matrix  $
D=\frac{1}{1+\alpha+\frac{\rho n }{\sqrt{mn}}}\left(
\begin{array}{cc}
 T_1 +\alpha e_m e_m^T& \rho e_m e_n^T \\
\rho e_n e_m^T& T_2+(\alpha+\rho\frac{n-m}{\sqrt{mn}})e_n e_n^T \\
\end{array}
\right)$ which is of order $(m+n)$ is doubly stochastic. Moreover, the $(m+n)$ eigenvalues of $D$ are given by:\\ $1,
\frac{(1+\alpha)\sqrt{mn}-\rho m}{(1+\alpha)\sqrt{mn}+\rho n},
\frac{1}{1+\alpha+\frac{\rho n }{\sqrt{mn}}}\lambda_2,..,
\frac{1}{1+\alpha+\frac{\rho n }{\sqrt{mn}}}\lambda_m,
\frac{1}{1+\alpha+\frac{\rho n }{\sqrt{mn}}}\mu_2,..,
\frac{1}{1+\alpha+\frac{\rho n }{\sqrt{mn}}}\mu_n.$
\end{theorem}

\begin{proof} A direct verification shows that the matrix $D$ is doubly stochastic. Now Applying
 Theorem 2.3 with $k=2,$ $n_1=m$ and  $n_2=n,$ and taking
 $A_1= T_1,$ $A_2=T_2,$  $\rho_{12}=\rho_{21}=\rho,$
  $\rho_{11}=\alpha $ and $\rho_{22}=\alpha+\rho\frac{n-m}{\sqrt{mn}}$ in the matrix $B$ of $(3)$
  and then the result follows easily.
\end{proof}\\
We conclude this section by noting that in the preceding two
theorems if $T_1$ and $T_2$ are symmetric then $D$ is necessarily
symmetric. Again the preceding two results are useful in the study of the inverse
eigenvalue problem for symmetric doubly stochastic matrices (see the techniques used in
\cite{mourad}).

\section{Application to graph theory}

In this section, we present some applications of the results
of Section 2 in graph theory. The first one deals with graph
spectra which has been an object of study for a while (see~\cite{br,cv}). The second
application is concerned with graph energy where the problem of
characterizing the set of positive numbers that can occur as
energy of a certain graph has been an active area of research in
recent years (see~\cite{ba,gu,nik}). First let us introduce some more notation. A simple
graph $G$ is a pair of sets $(V(G);E(G))$ such that $V$ is a
nonempty finite set of $n$ vertices and $E$ is the set of $m$
edges with no loops nor multiple edges. The adjacency matrix of a
graph $G$ will be denoted by $A(G)$ and its eigenvalues
$\lambda_1,...\lambda_n$ form the spectrum of $G$ which will be
denoted by $\sigma(G).$ The energy of a graph $G$ is defined as
$E(G)=\sum_{i=1}^{n}|\lambda_i|$ where $\lambda_1,...\lambda_n$
are the eigenvalues of $G.$ More generally, the energy of any
square matrix is defined analogously.

 The complete $k$-partite
graph $K_{n_1,n_2,...,n_k}$ has vertices partitioned into $k$
subsets of $n_1, n_2,...,n_k$ elements each, and two vertices are
adjacent if and if only if they are in different subsets in the
partition. For our purposes, we need the following well-known results
 (see~\cite{ba} for example).
\begin{theorem}
Let $A$ and $B$ be the adjacency matrices of the graphs $G$ and
$H$ respectively. Then $G$ is isomorphic to $H$ if and only if
$B=PAP^T$ for some permutation matrix $P.$
\end{theorem}
\begin{theorem}
Let $G$ be  a $d$-regular graph with $n$ vertices. Then the
spectral radius of $G$ equals $d$ and it is an eigenvalue of $G$
with corresponding unit eigenvector equals to
$e_{n}=\frac{1}{\sqrt{n}}(1,1,...,1)^T.$
\end{theorem}

For $j = 1, 2, . . . ,k,$ let $A_j$ be the adjacency matrix of a
$d_j$-regular graph $G_j.$ Let $\widehat{G}$ be {\it graph join} of $G_1,G_2,...,G_k$ i.e. the graph obtained
from $G_1,G_2,...,G_k$ by connecting each vertex of $G_i$ to every
vertex of $G_j$ for all $1\leq i, j \leq k$ with $i\neq j.$ Now
taking $\rho_{jj}=0$ for all $j = 1, 2, . . . ,k$ and
$\rho_{ij}=\sqrt{n_in_j}$ for all $1\leq i,j \leq k$ with $i\neq
j$ in the matrix $B$ of (3), then the matrix $B$ has as diagonal
blocks $A_1,A_2,...,A_k$ and each of its remaining entries is $1.$
In other words,  $B$ becomes the adjacency matrix of the graph
$\widehat{G}$ and the diagonal entries of the matrix $\widehat{B}$
are $d_1,d_2,...,d_k$ and all the remaining entries of
$\widehat{B}$ are equal to $\sqrt{n_in_j}.$ Thus in view of
Theorem 2.3 and Theorem 5.2, we have the following.
\begin{theorem}
For each $j = 1, 2, . . . ,k,$ let $G_j$ be a $d_j$-regular graph
with $n_j$ vertices and spectrum $\sigma(G_j).$  Let $\widehat{G}$
be the graph obtained from $G_1,G_2,...,G_k$ by connecting each
vertex of $G_i$ to every vertex of $G_j$ for all $1\leq i, j \leq
k$ with $i\neq j.$ In addition, let $\rho_{jj}=0$ for all $1\leq j
\leq k$ and let $\rho_{ij}=\sqrt{n_in_j}$ for all $1\leq i, j \leq
k$ with $i\neq j$ in the matrix $B$ of $(3)$ as well as in
$\widehat{B}$ of $(4).$ Then the spectrum of $\widehat{G}$ is
given by:
$$\sigma(\widehat{G})=\overset{k}
{\underset{j=1}{\bigcup}}(\sigma(G_{j})\backslash\{d_{j}\})\bigcup\sigma(\widehat{B}),$$
and its energy is given by:
$$E(\widehat{G})=\sum_{j=1}^{k}(E(G_{j})-d_j)+E(\widehat{B}).$$
\end{theorem}

Now if for each $j = 1, 2, . . . ,k,$  $G_j$ is a graph with $n_j$
vertices and with no edges i.e. $A_j$ is the zero matrix of order
$n_j,$ then clearly $\widehat{G}=K_{n_1,n_2,...,n_k}.$ Thus we
have the following result.
\begin{corollary}
Consider the complete $k$-partite graph $K_{n_1,n_2,...,n_k}$ for
some positive integers $n_1,n_2,...,n_k$ and let $X=(x_{ij})$ be
the $k\times k$ matrix whose diagonal entries are all zeroes and
whose off-diagonal entries are $x_{ij}=\sqrt{n_in_j}.$   Then
$\sigma(K_{n_1,n_2,...,n_k})=\sigma(X)\bigcup\{0,...,0\},$ and
hence $E(K_{n_1,n_2,...,n_k})=E(X).$
\end{corollary}

 Let $G_1$ be a regular graph with $n$ vertices and let $G_2,G_3,...,G_k$ be $k-1$
  graphs that are all isomorphic to $G_1,$ and
define the graph $\widehat{G}$ as earlier. Then in view of Theorem
5.1, we have the following.
\begin{theorem}
 let $G_1$ be a $d$-regular graph with $n$ vertices and $A$ be its adjacency matrix. Let $\sigma(G_1)=(d;\lambda_2,...,\lambda_n)$ with
 $d\geq\lambda_2\geq ...\geq \lambda_n.$ Let
 $G_2,...,G_k$ be $k-1$ graphs which are all isomorphic to $G_1$.  In addition, let $\widehat{G}$ be the graph
  obtained from $G_1,G_2,...,G_k$ by connecting each vertex of $G_i$ to every vertex of $G_j$ for all
 $1\leq i, j \leq k$ with $i\neq j.$ Then the spectrum of
 $\widehat{G}$ is the $nk$-tuple:
 $$\sigma(\widehat{G})=\left( \underbrace{d+n(k-1),\lambda_2,...,\lambda_n},\underbrace{d-n,\lambda_2,...,
 \lambda_n},...,\underbrace{d-n,\lambda_2,...,\lambda_n}\right)$$
 and then its energy is $$E(\widehat{G})=d+n(k-1)+(k-1)\left\vert d-n\right\vert +\sum_{j=2}^{k}k |\lambda_j|.$$
\end{theorem}

\begin{proof} It suffices to notice that the adjacency matrix of
the graph $\widehat{G}$ is the matrix $B$ of (3) whose diagonal
blocks are $A, P_2AP_2^T,...,P_kAP_k^T$ for some permutation matrices
$P_2,...,P_k$ and all the remaining entries are equal to $1.$
Then clearly $\widehat{B}$ is the $k\times k$ matrix whose each diagonal
entry is  $d$ and whose off-diagonal entries are all equal
to $\sqrt{n^2}=n,$  and the proof is complete.
\end{proof}\\

As a conclusion, we obtain the following.
\begin{corollary}
Let $G_1$ be a graph with $n$ vertices and no edges and let
$G_2,...,G_k$ be $k-1$ graphs which are all isomorphic to $G_1.$
Suppose that $\widehat{G}$ is the graph
  obtained from $G_1,G_2,...,G_k$ by connecting each vertex of $G_i$ to every vertex of $G_j$ for all
 $1\leq i, j \leq k$ with $i\neq j,$
then $\widehat{G}= K_{n,n,...,n}$ $($with $n$ repeated $k$
times$)$ and $E(\widehat{G})=2n(k-1).$
\end{corollary}
\begin{proof} It suffices to apply the preceding theorem with
$d=\lambda_2=...=\lambda_n=0.$
\end{proof}\\

 In a similar way to what was mentioned earlier in Section 3, the choice of the nonnegative numbers  $\rho_{ij}$
  plays an important role in constructing new graphs from old ones and in the study of
   the relations between the spectra and energies of the old and new graphs. More precisely,
taking $\rho_{jj}=0$ for all $j = 1, 2, . . . ,k$ and some of the
$\rho_{ij}=0$ for $i\neq j$ in the matrix $B$ of (3), we can
obtain many new consequences that are similar in nature to the preceding
results. More specifically, these results deal with constructing
other types of graphs and study their spectra and their energies
in the same way we defined $\widehat{G}$ and studied its spectrum
$\sigma(\widehat{G})$ and its energy $E(\widehat{G}).$ As an
illustration, we consider the case where $\rho_{ij}=0$ for all $i=j$ and for all
$|i-j|> 1$ and $\rho_{jj+1}=\sqrt{n_jn_{j+1}}$ otherwise. Then
 in this case, the matrix $B$ becomes the matrix $C$ of (1) and
then $\widehat{B}=\widehat{C}.$ Thus we clearly obtain the
following result which appears in \cite{iv}.
\begin{theorem}(\cite{iv})
For each $j = 1, 2, . . . ,k,$ let $G_j$ be a $d_j$-regular graph
with $n_j$ vertices and spectrum $\sigma(G_j).$  If $\tilde{G}$ is
the graph obtained from $G_1,G_2,...,G_k$ by connecting each
vertex of $G_j$ to every vertex of $G_{j+1},$ then
$$E(\tilde{G})=\sum_{j=1}^{k}(E(G_{j})-d_j)+E(\widehat{C}).$$
\end{theorem}

By taking all the $G_i$ isomorphic,  we obtain the following.
\begin{corollary}
Let $G_1$ be a $d$-regular graph with $n$ vertices  and spectrum
$\sigma(G_1)=(d;\lambda_2,...,\lambda_n).$  Let $G_2,...,G_k$ be
$k-1$ graphs which are all isomorphic to $G_1.$ If $\tilde{G}$ is
the graph obtained from $G_1,G_2,...,G_k$ by connecting each
vertex of $G_j$ to every vertex of $G_{j+1}$ for all $1\leq  j
\leq k,$ then $$E(\tilde{G})=\sum_{j=1}^{k} \left\vert
d+2n\cos\frac{j\pi}{k+1}\right\vert +\sum_{j=2}^{k}k
|\lambda_j|.$$
\end{corollary}
\begin{proof}
Clearly, the adjacency matrix of the graph
$\tilde{G}$ is the matrix $C$ of (1) whose diagonal blocks are $A,
P_2AP_2^T,...,P_kAP_k^T$ for some permutation matrices
$P_2,...,P_k$ and all the blocks above and below the diagonal
blocks have each entry equals to $1.$ Therefore $\widehat{C}$ is
the $k\times k$ tridiagonal matrix whose diagonal entries are all
$d$ and whose entries below and above the diagonal are all equal
to $\sqrt{n^2}=n$ that is $\widehat{C}=dI_k+nA(P_k)$ where $P_k$
is the path of order $k$ and $A(P_k)$ is its adjacency matrix. Now
by \cite{ba} for example,  the eigenvalues of $A(P_k)$ are
$(2\cos\frac{\pi}{k+1},2\cos\frac{2\pi}{k+1},...,2\cos\frac{k\pi}{k+1}).$
Therefore
$\sigma(\widehat{C})=(d+2n\cos\frac{\pi}{k+1},d+2n\cos\frac{2\pi}{k+1},...,d+2n\cos\frac{k\pi}{k+1}),$
and then the proof can be easily completed.
\end{proof}\\
\begin{corollary}
Let $G_1$ be a graph with $n$ vertices and no edges and let
$G_2,...,G_k$ be $k-1$ graphs which are all isomorphic to $G_1.$
If $\tilde{G}$ is the graph obtained from $G_1,G_2,...,G_k$ by
connecting each vertex of $G_j$ to all vertices of $G_{j+1}$ for
all $1\leq  j \leq k,$ then
$E(\tilde{G})=E(P_k)=\sum_{j=1}^{k}\left\vert
2n\cos\frac{j\pi}{k+1}\right\vert.$
\end{corollary}

\begin{proof} Apply the preceding theorem with
$d=\lambda_2=...=\lambda_n=0.$
\end{proof}

\end{document}